\documentclass[11pt]{amsart}

\usepackage{amsmath, amssymb, amsfonts, mathrsfs, mathtools}
\usepackage{latexsym, esint, graphicx, xcolor, color}
\usepackage{wasysym, stmaryrd}
\usepackage{hyperref}
\usepackage[alphabetic]{amsrefs}
\usepackage{geometry}

\hypersetup{
    linkbordercolor = {white}, 
    citebordercolor = {white},
}

\numberwithin{equation}{section}

\newtheorem{thm}[equation]{Theorem}
\newtheorem{conj}[equation]{Conjecture}

\newtheorem{lma}[equation]{Lemma}
\newtheorem{prop}[equation]{Proposition}

\newtheorem{claim}[equation]{Claim}
\theoremstyle{remark}
\newtheorem{rem}[equation]{Remark}

\newcommand{\set}[1]{\left\{#1\right\}}

\newcommand{\pr}[1]{\left(#1\right)}

\newcommand{\bb}[1]{\mathbb{#1}}

\def\be{\begin{equation}}
\def\ee{\end{equation}}
\def\beq{\begin{eqnarray}}
\def\eeq{\end{eqnarray}}
\def\bee{\begin{equation*}}
\def\eee{\end{equation*}}

\def\Ric{\operatorname{Ric}}
\def\Rm{\text{\rm Rm}}

\def\sbst{\subseteq}
\def\e{\varepsilon}

\def\a{{\alpha}}
\def\b{{\beta}}

\def\p{\partial}

\def\heat{\left(\frac{\p}{\p t}-\Delta_{g(t)}\right)}

\def\cS{\mathcal{S}}

\begin{document}

\title{Local mollification of metrics with small curvature concentration}

 \author{Man-Chun Lee}
\address[Man-Chun Lee]{Department of Mathematics, The Chinese University of Hong Kong, Shatin, Hong Kong, China}
\email{mclee@math.cuhk.edu.hk}

 \author{Tang-Kai Lee}
\address[Tang-Kai Lee]{Department of Mathematics, Columbia University, New York, NY, 10027.}
\email{lee.tang-kai@columbia.edu}


\date{\today}

\begin{abstract}
In this work, we establish a local smoothing result on metrics with small curvature concentration with respect to Sobolev constants and volume growth. 
As an application, we prove that the Gromov-Hausdorff limit of compact manifolds with bounded curvature concentration, Ahlfors $n$-regularity, and bounded Sobolev constants is a topological manifold away from finitely many points. 
In the complete non-compact case, we show that manifolds with Euclidean-type Sobolev inequality, Euclidean volume growth, and small curvature concentration are necessarily diffeomorphic to Euclidean spaces. 
In contrast with all previous works, we remove the Ricci curvature condition. 
\end{abstract}

\maketitle

\markboth{Man-Chun Lee, Tang-Kai Lee}{Local mollification of metrics with small curvature concentration}

\section{\bf Introduction}\label{sec: introduction}

Given a complete manifold $(M,g),$ 
one of the general themes in Riemannian geometry is to characterize the topology of $M$ under curvature conditions on the metric $g$. 
An important theorem of Cheeger \cite{Cheeger1970} shows that there are only finitely many diﬀeomorphism types for non-collapsed compact manifolds with bounded curvature and diameter. 
There have been many efforts in generalizing Cheeger’s finiteness theorem to weaker curvature assumptions. 
Relaxing the pointwise control $||\Rm||_{L^\infty}$ to an integral bound $||\Rm||_{L^p}$  has attracted much attention. 
The case $p=n/2,$ which we call {\it bounded curvature concentration}, is particularly interesting since the integral is scaling invariant without being weighted by the volume. 
The integral bound of curvature also appears naturally in many geometric problems. 
Perhaps more importantly, integral curvature bounds naturally arise in studying singular spaces as a weak formulation of curvature. 
For instance, it plays an important role in the RCD theory and positive mass theorems on asymptotically flat manifolds with  singularities. 
One naturally asks what geometric and topological results can be extended to integral curvature bounds. 

There have been many interesting generalizations of comparison geometry in this direction; for instance, see \cites{PetersenWei1,PetersenWei,Carron2019} and the references therein.  
An integral curvature bound is however not strong enough to constrain the topology in general. 
Indeed, it was shown by Yang~\cite{DeaneYang} that the exact analogy of Cheeger's finiteness cannot be true in case $p<+\infty$, due to the failure of controlling the volume locally. 
There is a gap phenomenon, as already observed in \cites{Gallot,DeaneYang,PetersenWei1,PetersenWei}; 
see \cite{DaiWei}*{Section 5} for a survey of more results in this direction.

In this work, we are particularly interested in the gap phenomenon and compactness under the scaling invariant curvature control $||\Rm||_{L^{n/2}(M)}<+\infty$.  
In Ricci geometry, this condition has been studied extensively in the literature; see \cites{Anderson1,Anderson,AndersonCheeger,Bando,Nakajima1988} for examples. 
In \cite{Cheeger2003}, Cheeger systemically studied the singularities of the Gromov--Hausdorff limit of non-collapsed manifolds with Ricci lower bounds and integral curvature control. 
Particularly when $p=n/2$, Cheeger showed that the singularities must be isolated and the tangent cone at a singular point must be the cone over a quotient of the round sphere. 
One of the key ingredients is to prove an epsilon regularity result for non-collapsed manifold with small curvature concentration. 
The method of Cheeger is elliptic, based on harmonic functions.

It is natural to ask if a metric with sufficiently small integral curvature can be perturbed to a metric with pointwise curvature bound in suitable topology, in order to study geometric compactness with only an integral curvature bound. 
To the best of our knowledge, the smoothing by Ricci flow was first considered by Yang \cite{DeaneYang}. 
The Ricci flow is a one-parameter family of metrics $g(t)$ on a manifold such that 
$$\partial_t g(t)=-2\Ric(g(t)).$$
Viewed as a nonlinear geometric heat equation, the flow has been proven useful to regularize metrics and obtain striking consequences in geometry and topology. 
When both $||\Rm||_{L^{n/2}}$ and $||\Ric||_{L^p},p>n/2,$ are small, Yang \cite{DeaneYang} used the Ricci flow to show that the metric can be perturbed in the $C^0$ topology to a metric with bounded geometry. 
However, the $C^0$ perturbation is not expected in general if one only has a scaling invariant curvature bound, which can be seen from the point of view of tangent cones.

In \cite{ChanHuangLee2024}, the first named author with Chan and Huang considered a parabolic approach to Cheeger's epsilon regularity theorem.  
In short, non-collapsed metrics with Ricci lower bounds and small local curvature concentration can be regularized by Ricci flow for a uniform short time with instantaneous bounded geometry. 
The perturbation is in the (pointed) Gromov--Hausdorff topology, in the presence of Ricci lower bounds by the distance estimate of \cite{TianWangJAMS}.
This improved and sharpened the result \cite{DeaneYang} of Yang.
This is based on constructing distance like functions with critical integral Hessian estimates, using comparison geometry, so as to construct approximating Ricci flow solutions with suitable control. 
This smoothing result was later utilized by the first named author and Chan \cite{ChanLee2024} to prove a metric gap theorem of manifolds with $\Ric\geq 0$, Euclidean volume growth, and small curvature concentration; see also \cite{Martens2025}. 
The proof in \cite{ChanHuangLee2024} relies on Schoen--Yau's construction of proper functions which uses Bochner's formula for eigenfunctions. 
The Ricci curvature lower bound is crucial there in obtaining gradient estimates.

In this work, we further advance the results in \cite{ChanHuangLee2024, Martens2025-2} and show that the local smoothing holds in a greater generality.
In particular, we show that the assumption of Ricci curvature lower bounds of any kind is redundant. 

\begin{thm}\label{thm:localRF}
For $n\geq 3$ and $v_0,C_s>0$, there exist $\delta_1(n,v_0,C_s),\Lambda_1(n,v_0,C_s),L(n,v_0,C_s)$, $r_1(n,v_0,C_s),$ and $ \hat T(n,v_0,C_s)>0$, such that the following holds. 
Let $(M^n,g_0)$ be a manifold such that $B_{g_0}(x_0,3+\bar r)\Subset M$ for some $(x_0,\bar r)\in M\times [0,+\infty)$.
Suppose for all $x\in B_{g_0}(x_0,2+\bar r),$
\begin{enumerate}
\item[(a)] $ \mathrm{Vol}_{g_0}\left(B_{g_0}(x,r)\right)\leq v_0 r^n$ for all $r\in (0,1]$;
\item[(b)] for all $u\in W^{1,2}_0( B_{g_0}(x,1), g_0)$,
$$\left(\int_{ B_{g_0}(x,1)} u^\frac{2n}{n-2}\,d\mathrm{vol}_{g_0}\right)^\frac{n-2}{n}\leq C_s \int_{ B_{g_0}(x,1)} \left(|\nabla u|^2+u^2 \right)\, d\mathrm{vol}_{g_0};$$
\item[(c)] $||\Rm(g_0)||_{L^{n/2}(B_{g_0}(x,1),g_0)}\leq \e$  for some $\e\in (0,\delta_1]$. 
\end{enumerate}
Then there exists a smooth Ricci flow solution $g(t)$ defined on $B_{g_0}(x_0,1+\bar r)\times [0,\hat T]$ such that $g(0)=g_0$ and 
\begin{enumerate}
\item $ |\Rm(g(t))|\leq \Lambda_1  t^{-1}(\e+t^{2/n})\leq \frac1{4(n-1)}t^{-1}$,
\item $\mathrm{inj}(g(t))\geq \sqrt{\Lambda_1^{-1} t}$,
\item $ ||\Rm(g(t))||_{L^{n/2}(B_{g(t)}(x,1/2))}\leq \Lambda_1 (\e +t^{2/n}) $, and
\item $\nu \left(B_{g(t)}(x,2L\sqrt{\hat T},g(t), L^2\hat T \right)\geq -\Lambda_1$
\end{enumerate}
for all $(x,t)\in B_{g_0}(x_0,1+\bar r)\times (0,\hat T]$. Moreover, for all $x,y\in B_{g_0}(x_0,r_1)$ and $ t\in [0,\hat T]$,
$$ \Lambda_1^{-1} \left[d_{g(t)}(x,y)\right]^\frac{1}{1-2(n-1)\Lambda_1\e}\leq d_{g_0}(x,y)\leq \Lambda_1 \left[d_{g(t)}(x,y)\right]^\frac{1}{1+2(n-1)\Lambda_1\e}.$$
\end{thm}

\begin{rem}
The entropy lower bound measures the optimal constant of the log-Sobolev inequality. 
In terms of boundedness, this is equivalent to a (weighted) Sobolev constant. 
See \cites{ChanChenLee2022,Ye2015} for example. 
\end{rem}

From the instantaneous curvature bound of $g(t)$, Theorem~\ref{thm:localRF} provides a local coordinate for $g_0$ quantitatively such that the distance is locally bi-H\"older equivalent to the Euclidean space.
The smoothing is purely local in nature and does not require any local pointwise curvature bound. We use it to consider Ahlfors $n$-regular compact manifolds with bounded diameters, Sobolev constants and curvature concentration. 
We use the local smoothing to show that the Gromov--Hausdorff limit of a sequence of manifolds in this collection must be a topological manifold away from finitely many points. 

\begin{thm}\label{intro:main-2} 
Let $n\geq 3$ and $C_s,v_0,\Lambda_0,D_0>0$. Suppose $(M^n_i,g_i)$ is a sequence of compact manifolds such that  
\begin{enumerate}
\item[(a)] for all $u\in W^{1,2}(M_i)$, 
$$\left(\int_{M_i} u^\frac{2n}{n-2}\,d\mathrm{vol}_{g_i} \right)^\frac{n-2}{n}\leq C_s\int_{M_i} \left(|\nabla u|^2+ u^2 \right) \,d\mathrm{vol}_{g_i};$$
\item[(b)] $\mathrm{Vol}_{g_i}\left(B_{g_i}(x,r) \right)\leq v_0r^n$ for all $(x,r)\in M_i\times (0,1]$;
\item[(c)] $||\Rm(g_i)||_{L^{n/2}(M_i)}\leq \Lambda_0$;
\item[(d)] $\mathrm{diam}(M_i)\leq D_0$.
\end{enumerate}
Then there exists a compact metric space $(X,d_\infty)$ and $\mathcal{S}=\{p_i\}_{i=1}^N$ such that $(M_i,g_i)$ converges subsequentially to $(X,d_\infty)$ in the Gromov--Hausdorff topology and $X\setminus \mathcal{S}$ is a topological manifold. 
\end{thm}

In the presence of Ricci lower bounds, the Sobolev constant can be ensured by volume non-collapsing. 
Thus, this theorem can be viewed as a generalization of results in Ricci geometry; see the compactness results in \cites{Anderson1,Anderson,Bando}.  Particularly, we extend the related compactness to cases without Ricci control. When $\Lambda_0$ is sufficiently small, the singular set $\mathcal{S}$ is empty and the limit $X$ admits a smooth structure by the Ricci flow smoothing in Theorem~\ref{thm:localRF}, see \cite[Theorem 1.4]{ChanChenLee2022}.
It is, however, unclear to us if this is the case in general; see Remark~\ref{rem:X-smooth-mfd}. 
In the presence of Ricci curvature lower bounds, it is also known that the tangent cone at $p\in \mathcal{S}$ is given by $C(\mathbb{S}^{n-1}/\Gamma)$ by the deep work of Cheeger \cite{Cheeger2003}. 
This plays an important role in the recent work of Jiang--Wei \cite{JiangWei2025} on the finiteness of diffeomorphisms. 
It will be interesting to know if the finite diffeomorphism theorem in \cite{JiangWei2025} can be generalized to the case without the Ricci curvature assumption. 
This is related to identifying the asymptotic behavior near a singular point $p\in \mathcal{S}$.

Another application of Theorem~\ref{thm:localRF} is about complete non-compact geometry and is related to the case without singularities. 
We consider the topological gap phenomenon of complete non-compact manifolds with small curvature concentration. 
The problem has been studied by many authors; for example, see \cites{MR1484888, Ledoux,Xia1,Xia2, Cheeger2003, Chen2022, ChanHuangLee2024, ChauMarten2026} and the references therein. 
We show that a manifold with small curvature concentration is diffeomorphic to the Euclidean space provided it supports a Euclidean-type Sobolev inequality and has Euclidean volume growth. 
This does not rely on the Ricci curvature and thus  extends the results in \cites{ChanHuangLee2024,Cheeger2003}.

\begin{thm}\label{intro:main}
For $n\geq 3$ and $C_s,v_0>0$, there exists $ \delta_0(n,C_s,v_0)>0$  such that the following is true. 
Suppose $(M^n,g_0)$ is a complete non-compact manifold such that  
\begin{enumerate}
\item[(a)] for all $u\in W^{1,2}_0(M)$, 
$$\left(\int_M u^\frac{2n}{n-2}\,d\mathrm{vol}_{g_0} \right)^\frac{n-2}{n}
\leq C_s \int_M |\nabla u|^2\,d\mathrm{vol}_{g_0};$$
\item[(b)] $\mathrm{Vol}_{g_0}\left(B_{g_0}(x,r) \right)\leq v_0r^n$ for all $(x,r)\in M\times \mathbb{R}_+$;
\item[(c)] $||\Rm(g_0)||_{L^{n/2}(M)}\leq  \delta_0(n,C_s,v_0)$.
\end{enumerate}
Then $M$ is diffeomorphic to $\mathbb{R}^n$.
\end{thm}

The Sobolev condition plays a strong role in non-collapsing. 
It is also easily seen from the Eguchi-Hanson metric that smallness on $||\Rm||_{L^{n/2}}$ is necessary. 
It is, however, unclear if this is still necessary when $n$ is odd.

When the negative part of Ricci curvature has a stronger decay, the volume growth condition holds automatically by the work of Carron \cite{Carron2019}*{Theorem B}. 
It will be interesting to know if it is necessary to assume (b) a-priori in Theorem~\ref{intro:main}; 
cf. \cite{ChauMarten2024} for the case of bounded curvatures.

A natural collection to consider is the class of Euclidean submanifolds.
Based on the Sobolev inequality of Michael--Simon~\cite{MichaelSimon}, a complete Euclidean submanifold with small total curvature satisfies the Euclidean-type Sobolev inequality.
Thus, Theorem~\ref{intro:main} implies that such a submanifold is diffeomorphic to $\bb R^n$ if its volume growth is Euclidean.
Moreover, when the submanifold is a convex hypersurface in $\bb R^{n+1},$ one can remove the volume growth assumption and use the main theorem in \cite{ChanLee2024} to show that the hypersurface is isometric to $\bb R^n$ when it has small total curvature.
In general, in Theorem~\ref{intro:main}, when the scalar curvature is, furthermore, assumed to be non-negative, we expect to get a stronger rigidity result as in the Eucludean case.
See Remark~\ref{rem:positive-R}.

We organize the paper as follows.
Section~\ref{sec:RF-prel} provides the necessary a priori estimates that we will need in Section~\ref{sec:RF-exist}, in which the main local smoothing result (Theorem~\ref{thm:localRF}) is proven.
We prove the compactness theorem (Theorem~\ref{intro:main-2}) in Section~\ref{sec:compactness} and the topological gap theorem (Theorem~\ref{intro:main}) in Section~\ref{sec:gap}.

{\it Acknowledgment:} The authors would like to thank P.-Y. Chan for fruitful discussion. M.-C. Lee is  supported by Hong Kong RGC grant (Early Career Scheme) of Hong Kong No. 24304222 and No. 14300623, NSFC grant No. 12222122 and an Asian Young Scientist Fellowship.
T.-K. Lee thanks the support from Bill Minicozzi through the NSF Award DMS-2304684 and the support from the US Junior Oberwolfach Fellowship through the NSF Award DMS-2230648 when the project is carried out.

\section{\bf  A priori estimates of Ricci flow}\label{sec:RF-prel}

In this section, we will derive some a-priori estimates of the Ricci flow, which is related to the curvature concentration.

Along the Ricci flow, it is more common to consider the log-Sobolev constant, which is measured by the $\nu$-entropy. 
It is well-known that the $\nu$-entropy is monotone along a complete Ricci flow with bounded curvature thanks to the celebrated work of Perelman \cite{Perelman2002}; see also the complete non-compact case by Chau--Tam--Yu \cite{ChauTamYu2011}. 
We need a localized version by Wang \cite{Wang2018}. 
Following~\cite{Wang2018}, for a connected domain $\Omega$ with possibly empty boundary in $M$,  we define
\begin{equation*}
\left\{
\begin{array}{ll}
&D_g(\Omega):=\left\{u\in W^{1,2}_0(\Omega): u\geq 0 \text{  and  } \|u\|_{L^2(\Omega)}=1 \right\},\\[5mm]
&W(\Omega, g, u, \tau):=\displaystyle\int_{\Omega} \pr{\tau(\mathcal{R}\cdot u^2+4|\nabla u|^2)-2u^2\log u} \;   d\mathrm{vol}_g -\frac{n}{2}\log(4\pi\tau)-n,
\\[5mm]
&\nu(\Omega, g, \tau):= \inf\limits_{u\in D_g(\Omega), s \in (0, \tau]}W(\Omega, g, u, s)
\end{array}
\right.
\end{equation*}
where $\mathcal{R}$ denotes the scalar curvature.

The local entropy $\nu(\Omega,g,\tau)$ measures the optimal (weighed) log-Sobolev constant of $\Omega$ in scale of $\tau$ with respect to $g$. 
This also controls the Sobolev constant in some suitable scale by \cite{Martens2025}*{Lemma 3.3} which is built on \cite{Ye2015}. 
If $g(t)$ is a complete Ricci flow on $\Omega$ with bounded curvature, it is well-known that $\nu(\Omega,g(t),\tau-t)$ is monotonically non-decreasing in $t$ by the celebrated work of Perelman \cite{Perelman2002}.

Although the (almost) monotonicity of the (localized) $\nu$-entropy does not necessarily hold for incomplete local Ricci flow solutions, we will construct local solutions in a way that the entropy is still under control.

In this section, we will derive a priori  estimates, assuming a bound on the localized entropy. 
We first need a local persistence of local curvature concentration, which is an essential key to the semi-preservation of small curvature concentration, see \cite{ChanChenLee2022,ChanHuangLee2024,ChanLee2024,ChauMarten2024,ChauMarten2026,Martens2025,Martens2025-2,Chen2022}. We need the following localized version of \cite{ChanLee2024}*{Lemma 2.2}.

\begin{lma}\label{lma:local-total-pre}
For all $A,\a>0$, there exist $\delta_1(n,A)\in (0,1)$, $\tilde T(n,\a),C_1(n)>0$ such that the following is true. 
Suppose $g(t)$ is a smooth solution to the Ricci flow on $M\times [0,T]$ and $x_0\in M$ such that for all $t\in [0,T]$,
\begin{enumerate}
\item[(a)] $B_{g(t)}(x_0,2)\Subset M$; 
\item[(b)] $||\Rm(g(t))||_{L^{n/2}(B_{g(t)}(x_0,2),g(t))}<\delta_0$ for some $\delta_0\leq \delta_1$;
\item[(c)] $\nu(B_{g(t)}(x_0,2),g(t),1)\geq -A$;
\item[(d)] $\Ric(g(t))\leq \a (n-1)t^{-1}$.
\end{enumerate}
Then for all $t\in [0,T\wedge\tilde T]$, we have
\begin{equation*}
\int_{B_{g(t)}(x_0,1)}|\Rm(g(t))|^{n/2}\, d\mathrm{vol}_{g(t)} \leq \int_{B_{g_0}(x_0,2)}|\Rm(g_0)|^{n/2}\, d\mathrm{vol}_{g_0}+ C_1\delta_0^{n/2} t.
\end{equation*}
\end{lma}
\begin{proof}
Let $\varphi(s)$ be a smooth non-increasing function on $[0,+\infty)$ such that $\varphi\equiv 1$ on $[0,1]$, vanishes outside $[0,2],$ and satisfies $|\varphi'|^2\leq 10^3\varphi$ and $\varphi''\geq -10^3\varphi$. 
Let 
$$\eta(x,t):=d_{g(t)}(x,x_0)+2(\a+1)(n-1) \sqrt{t}$$ 
and define $\phi(x,t):=e^{-10^3t}\varphi\left(\eta(x,t) \right)$ for $x\in B_{g(t)}(x_0,2)$ and $t\in [0,T]$. 
It follows from \cite{Perelman2002}*{Lemma 8.3} that there exists $\tilde T=\tilde T(n,\a)>0$ such that 
\begin{equation*}
\heat \phi \leq 0
\end{equation*}
for $t\in [0,T\wedge \tilde T]$, in the barrier sense; see also \cite{SimonTopping2022}*{Lemma 7.1} for a detailed discussion. 
The inequality also holds in the distribution sense by \cite{MantegazzaMascellaniUraltsev2014}*{Appendix A}. 

Furthermore, $|\nabla\phi|^2\leq 10^3 \phi,$ so that \cite{ChanChenLee2022}*{Lemma 2.1} implies 
\begin{equation*}
\begin{split}
\frac{d}{dt} \int_M \phi^2 |\Rm|^{n/2}\,d\mathrm{vol}_{g(t)}&\leq -C_0^{-1}  \int_M |\nabla \left(\phi |\Rm|^{n/4} \right)|^2  d\mathrm{vol}_{g(t)}\\
&\quad +C_0\int_M \phi^2 |\Rm|^{n/2+1} \,d\mathrm{vol}_{g(t)}\\
&\quad +C_0\int_M \phi |\Rm|^{n/2} \,d\mathrm{vol}_{g(t)}\\
&:=-\mathbf{I}+\mathbf{II}+\mathbf{III}
\end{split}
\end{equation*}
for some dimensional constant $C_0(n)>0$. We note that the evolution equation in \cite{ChanChenLee2022}*{Lemma 2.1} is purely local, and thus the completeness of the solution there is not necessary.

Since $\phi(\cdot,t)$ is compactly supported on $B_{g(t)}(x_0,2)$ for each $t\in [0,T]$, 
\begin{equation*}
\mathbf{III}\leq C_0\delta_0^{n/2}.
\end{equation*}
Similarly, together with H\"older's inequality, we have
\begin{equation*}
\begin{split}
\mathbf{II}&\leq C_0||\Rm||_{L^{n/2}(B_{g(t)}(x_0,2))}\cdot \left( \int_M \phi^\frac{2n}{n-2}|\Rm|^\frac{n^2}{2(n-2)}\,d\mathrm{vol}_{g(t)}\right)^\frac{n-2}{n}\\
&\leq C_0\delta_0 \cdot \left( \int_M \phi^\frac{2n}{n-2}|\Rm|^\frac{n^2}{2(n-2)}\,d\mathrm{vol}_{g(t)}\right)^\frac{n-2}{n}.
\end{split} 
\end{equation*}

It remains to handle $\mathbf{I}$. For any $t\in [0,T\wedge \tilde T]$, we apply \cite{Martens2025}*{Lemma 3.3(2)} with $\Omega=B_{g(t)}(x_0,2)$, $g=g(t)$ and $\tau=1$ to deduce that for some $\Lambda_0(n,A)>0$, we have 
\begin{equation}\label{eqn:Sobo}
\left(\int_{B_{g(t)}(x_0,2)} |u|^\frac{2n}{n-2}\,d\mathrm{vol}_{g(t)} \right)^\frac{n-2}{n}\leq  \Lambda_0  \int_{B_{g(t)}(x_0,2)} \left(|\nabla u|^2+ u^2\right)  \,d\mathrm{vol}_{g(t)} 
\end{equation}
for all $u\in W^{1,2}_0(B_{g(t)}(x_0,2))$, provided that $\delta_0\leq \delta_1$ and $\delta_1$ is small enough depending on $A.$ 
Applying \eqref{eqn:Sobo} to $u:=\phi |\Rm|^{n/4}$ yields 
\begin{equation*}
\begin{split}
\mathbf{I}&\geq (C_0\Lambda_0)^{-1} \left( \int_M \phi^\frac{2n}{n-2}|\Rm|^\frac{n^2}{2(n-2)}\,d\mathrm{vol}_{g(t)}\right)^\frac{n-2}{n}-C_0^{-1}\int_M \phi^2 |\Rm|^{n/2}\,d\mathrm{vol}_{g(t)}.
\end{split}
\end{equation*}

Combining the estimates above, we conclude that if $\delta_0\leq \delta_1\leq (C_0^2 \Lambda_0)^{-1}$, then 
\begin{equation*}
\begin{split}
&\quad \frac{d}{dt} \int_M \phi^2 |\Rm|^{n/2}\,d\mathrm{vol}_{g(t)}\\
&\leq (C_0+C_0^{-1}) \delta_0^{n/2}+\left( C_0\delta_0-(C_0\Lambda_0)^{-1} \right) \cdot \left( \int_M \phi^\frac{2n}{n-2}|\Rm|^\frac{n^2}{2(n-2)}\,d\mathrm{vol}_{g(t)}\right)^\frac{n-2}{n}\\
&\leq C_1(n)\delta_0^{n/2}.
\end{split}
\end{equation*}
The assertion follows from integrating this from the initial time slice. 
This completes the proof.
\end{proof}

\begin{rem}\label{rem:size-delta1}
Here we implicitly assume $\delta_1$ is small enough so that assumption in \cite{Martens2025}*{(2), Lemma 3.3} is fulfilled.
\end{rem}

\bigskip

We now show that if we have a priori scaling invariant curvature decay in time, then the assumption (b) in Lemma~\ref{lma:local-total-pre} will be satisfied, provided that the initial metric is quantitatively Ahlfors $n$-regular. 

\begin{prop}\label{prop:improve-local-total-pre}
For all $A,v_0>1$, there exist $\tilde T(n)$, $\delta_0(n,v_0,A)>0$ and $\a_0(n,v_0,A) \in (0,1)$ such that the following holds. 
Suppose $g(t)$ is a smooth solution to the Ricci flow on $M\times [0,T]$ and $x_0\in M$ such that for all $t\in [0,T]$, we have
\begin{enumerate}
\item[(a)] $B_{g(t)}(x_0,2)\Subset M$;
\item[(b)] $|\Rm(g(t))|\leq \a_0 t^{-1}$ on $B_{g(t)}(x_0,2)$;
\item [(c)]
$ \nu\left(B_{g(t)}(x_0,2),g(t),1\right)\geq  -A$;
\item[(d)] $\mathrm{Vol}_{g_0}\left( B_{g_0}(x,r)\right)\leq v_0r^n$ for $x\in B_{g_0}(x_0,1)$ and $r\in (0,1]$;
\item[(e)] $||\Rm(g_0)||_{L^{n/2}(B_{g_0}(x_0,2))}<\delta_0$.
\end{enumerate}
Then for all $t\in [0,T\wedge \tilde T]$,
\begin{equation*}
\left(\int_{B_{g(t)}\pr{x_0,1}} |\Rm(g(t))|^{n/2} \,d\mathrm{vol}_{g(t)} \right)^{2/n}\leq \delta_1
\end{equation*}
where $\delta_1(n,A)$ is the constant from Lemma~\ref{lma:local-total-pre}.
\end{prop}
\begin{proof}
In the proof, we will use $C_i$ to denote any dimensional constant. For $t\in [0,T]$ and $x\in B_{g(t)}(x_0,1)$, we define $ \rho(x,t):= \frac23 \left(\frac32-d_{g(t)}(x,x_0)\right),$ which is $ \frac23$ of the $g(t)$-distance from $x$ to the boundary of $ B_{g(t)}(x_0,\frac32)$.

Let $\delta_0\in (0,\delta_1)$ be a constant to be specified. 
We claim that 
\begin{equation}\label{eqn:pt-picking}
\left(\int_{B_{g(t)}(x,\rho(x,t))} |\Rm(g(t))|^{n/2}\,d\mathrm{vol}_{g(t)}\right)^{2/n}<\delta_1
\end{equation}
for all $ x\in B_{g(t)}(x_0,\frac32)$ and $t\in [0,T]$ { if $T$ is smaller than a dimensional constant}.

By \cite{SimonTopping2022}*{Corollary 3.3}, for some $\tilde T(n)>0,$ $ B_{g(t)}(x_0,\frac32)\Subset B_{g_0}(x_0,2)$ for $t\in [0,T\wedge \tilde T]$ and  $\rho(x,t)\to 0$ as $ x\to \partial B_{g(t)}(x_0,\frac32)$. 
Therefore, if the claim \eqref{eqn:pt-picking} is not true, then there exists $t_1\in (0,T\wedge \tilde T]$ such that 
\begin{enumerate}
\item[(i)] 
$||\Rm(g(t))||_{L^{n/2}(B_{g(t)}(x,\rho(x,t)))}<\delta_1$ for all $t\in [0,t_1)$ and $ x\in B_{g(t)}(x_0,\frac32)$;
\item[(ii)]
$||\Rm(g(t_1))||_{L^{n/2}(B_{g(t_1)}(x_1,\rho(x_1,t_1)))}=\delta_1$ for some  $ x_1\in B_{g(t_1)}(x_0,\frac32 )$.
\end{enumerate}

\bigskip

We denote $ \rho_1:=\rho(x_1,t_1)\leq 1$ for convenience.

\begin{claim}\label{claim:lower-bdd}
There exists $C_0(n)>1$ such that for any $\b>1$, if $\a_0\leq C_0^{-2}\b^{-2}\delta_1$, then $\rho_1\geq \b \sqrt{t_1}$.
\end{claim}
\begin{proof}[Proof of Claim~\ref{claim:lower-bdd}]
Given $\b>1$, we show that $\rho_1\geq \b\sqrt{t_1}$ if $\a_0$ is sufficiently small. 
Suppose this is not the case, that is, $\rho_1\leq \b\sqrt{t_1}$.  
By (ii) and the assumption on curvature estimates, 
\begin{equation*}
\begin{split}
\delta_1^{n/2}&\leq (\a t_1^{-1})^{n/2} \cdot \mathrm{Vol}_{g(t_1)}\left(B_{g(t_1)}(x_1,\rho_1) \right)\\
&\leq (\a t_1^{-1})^{n/2} \cdot \mathrm{Vol}_{g(t_1)}\left(B_{g(t_1)}(x_1,\b \sqrt{t_1}) \right).
\end{split}
\end{equation*}
By the volume comparison using $\Ric(g(t))\geq -(n-1)\a t^{-1}$  and a simple scaling argument, we can estimate the volume from above as long as $\a_0\b^2\leq 1$, so that
\begin{equation*}
\begin{split}
\delta_1^{n/2}\leq \pr{\frac12C_0\a_0^{1/2}\b}^n,
\end{split}
\end{equation*}
which is impossible if $\a_0\leq C_0^{-2}\b^{-2}\delta_1$.
\end{proof}

On the other hand, for all $x\in B_{g(t_1)}(x_1,\rho_1)$ and $t\in [0,t_1]$,
\begin{equation*} 
\begin{split}
\frac32\rho(x,t)&=\frac32-d_{g(t)}(x_0,x)\\
&\geq \frac32-d_{g(t)}(x_0,x_1)-d_{g(t)}(x_1,x)\\
&\geq \frac32-d_{g(t_1)}(x_0,x_1)-d_{g(t_1)}(x_1,x) -C_1\sqrt{\a_0 t_1}\\
&\ge \rho_1-C_1\sqrt{t_1}
\end{split}
\end{equation*}
for some $C_1>1$, where we have used \cite{SimonTopping2022}*{Corollary 3.3}. 
Hence, by Claim~\ref{claim:lower-bdd}, we might assume $\rho(x,t)\geq \frac14 \rho_1$  for all $x\in B_{g(t_1)}(x_1,\rho_1)$ and $t\in [0,t_1]$, if $ \a_0\leq (4C_0C_1)^{-2}\delta_1$. 
Therefore, (i) implies
\begin{equation}\label{eqn:pt-pick-2}
\int_{B_{g(t)}\pr{x,\frac14\rho_1}} |\Rm(g(t))|^{n/2}\,d\mathrm{vol}_{g(t)} \leq \delta_1^{n/2}
\end{equation} 
for all $t\in [0,t_1]$ and $x\in B_{g(t_1)}(x_1,\rho_1)$. 

\bigskip

Fix $z\in  B_{g(t_1)}(x_1,\rho_1)$ and  consider the rescaled Ricci flow $\tilde g(t):=\lambda^2 g(\lambda^{-2}t),t\in [0,\lambda^2 t_1]$ where $\lambda  =8  \rho_1^{-1}$. 
By \eqref{eqn:pt-pick-2} and the assumptions, for $t\in [0,\lambda^2 t_1],$ the rescaled flow satisfies 
\begin{itemize}
\item $B_{\tilde g(t)}(z,2)=B_{g(\lambda^{-2}t)}(z,\frac14\rho_1)\Subset M$;
\item $||\Rm(\tilde g(t))||_{L^{n/2}(B_{\tilde g(t)}(z,2))}<\delta_1$;
\item $\nu\left(B_{\tilde g(t)}(z,2),\tilde g(t),1 \right)\geq -A;$
\item $|\Rm(\tilde g(t))|\leq \a_0 t^{-1}$.
\end{itemize}
Here we have used the monotonicity on the scale $\tau$ in $\nu$ entropy from \cite{Wang2018}*{Proposition 3.2} and $\lambda^{-2}=8^{-2}\rho_1^2\leq 1$. 

We apply Lemma~\ref{lma:local-total-pre} to $\tilde g(t)$ to conclude that 
\begin{equation*}
\begin{split}
\int_{B_{\tilde g(t)}(z,1)}|\Rm(\tilde g(t))|^{n/2}\, d\mathrm{vol}_{\tilde g(t)} 
&\leq \int_{B_{\tilde g_0}(z,2)}|\Rm(\tilde g_0)|^{n/2}\, d\mathrm{vol}_{\tilde g_0}+ C_2\delta_1^{n/2} t\\
&\leq \delta_0^{n/2}+ C_2\delta_1^{n/2} t
\end{split}
\end{equation*}
for all $t\in [0,\lambda^2t_1 \wedge \tilde T(n)]$. 
Rescaling it back, we conclude that 
\begin{equation}\label{eqn:local-cur-conc}
\begin{split}
\int_{B_{g(s)}(z,\frac18\rho_1)} |\Rm(g(s))|^{n/2}\,d\mathrm{vol}_{g(s)} &\leq \delta_0^{n/2}+ 64C_2\delta_1^{n/2} \rho_1^{-2} s\\
&=: \delta_0^{n/2}+ C_3\delta_1^{n/2} \rho_1^{-2} s
\end{split}
\end{equation} 
for all $s\in [0,t_1\wedge 8^{-2}\tilde T \rho_1^2]$ and $z\in B_{g(t_1)}(x_1,\rho_1)$. 
By Claim~\ref{claim:lower-bdd}, if 
$$ \a_0\leq \min\set{(8C_0)^{-2}\tilde T\delta_1, (4C_1C_0)^{-2}\delta_1},$$ 
then \eqref{eqn:local-cur-conc} holds at $s=t_1$.
\bigskip

We now use \eqref{eqn:local-cur-conc} to contradict (ii), the equality at $(x_1,t_1)$. 
By choosing a maximal collection of $\{z_i\}_{i=1}^N$ in $B_{g(t_1)}\pr{x_1,\frac{39}{40}\rho_1}$ such that $B_{g(t_1)}\left(z_i,\frac1{40}\rho_1\right)$ are pair-wise disjoint, we have a collection of points $\{z_i\}_{i=1}^N$ such that
\begin{equation}\label{eqn:net}
\coprod_{i=1}^N B_{g(t_1)}\left(z_i,\frac1{40}\rho_1\right)\subseteq B_{g(t_1)}(x_1,\rho_1) \subseteq  \bigcup_{i=1}^N B_{g(t_1)}\left(z_i,\frac18\rho_1\right).
\end{equation}
Then (ii) and \eqref{eqn:local-cur-conc} imply
\begin{equation}\label{eqn:bound-delta_1-delta_0}
\begin{split}
\delta_1^{n/2}&\leq N \left( \delta_0^{n/2}+ C_3\delta_1^{n/2} \rho_1^{-2}t_1  \right).
\end{split}
\end{equation}

\begin{claim}\label{claim:cover-num}
There exist $\gamma_1(n,A),N_0(n,A,v_0)>0$ such that if $\a_0\leq \gamma_1\delta_1$, then $N\leq N_0$.
\end{claim}

\begin{proof}[Proof of Claim~\ref{claim:cover-num}]
This follows from \cite[Lemma 2.1]{Martens2025-2}; see also the proof of \cite[Theorem 1.2]{ChanChenLee2022}. 
For readers' convenience, we include its proof.

By \cite{SimonTopping2022}*{Corollary 3.3} and Claim~\ref{claim:lower-bdd}, we might assume $B_{g(t_1)}(x_1,\rho_1)\subseteq B_{g_0}(x_1,2\rho_1)$ and $B_{g(t_1)}\left(z_i,\frac1{40}\rho_1\right)\sbst B_{g(s)}(z_i, \rho_1/4)$ for $s\in[0,t_1]$. 
From \eqref{eqn:net} and the assumption of Ahlfors $n$-regularity, we have
\begin{equation}\label{eqn:covering}
\begin{split}
\sum_{i=1}^N \mathrm{Vol}_{g_0}\left(B_{g(t_1)}\left(z_i,\frac1{40}\rho_1\right) \right)
&\leq \mathrm{Vol}_{g_0}\left(B_{g(t_1)}\left(x_1,\rho_1\right) \right)\\
&\leq  \mathrm{Vol}_{g_0}\left(B_{g_0}\left(x_1,2\rho_1\right) \right)\leq v_0(2\rho_1)^n.
\end{split}
\end{equation}

It remains to control each $\mathrm{Vol}_{g_0}\left(B_{g(t_1)}\left(z_i,\frac1{40}\rho_1\right) \right)$ from below. We compare it with volume at $t=t_1$ as follows.  
For $s\in [0,t_1]$ and any $i,$
\begin{equation*}
\begin{split}
&\quad \frac{d}{ds}\mathrm{Vol}_{g(s)}\left(B_{g(t_1)}\left(z_i,\frac1{40}\rho_1\right) \right)\\
&=-\int_{B_{g(t_1)}\left(z_i,\frac1{40}\rho_1\right)}\mathcal{R}(g(s)) \,d\mathrm{vol}_{g(s)}\\
&\leq \left(\int_{B_{g(t_1)}\left(z_i,\frac1{40}\rho_1\right)}|\Rm(g(s))|^{\frac{n}{2}} \,d\mathrm{vol}_{g(s)}\right)^{2/n}\left[\mathrm{Vol}_{g(s)}\left(B_{g(t_1)}\left(z_i,\frac1{40}\rho_1\right) \right) \right]^{1-\frac2n}\\
&\leq \delta_1 \left[\mathrm{Vol}_{g(s)}\left(B_{g(t_1)}\left(z_i,\frac1{40}\rho_1\right) \right) \right]^{1-\frac2n}
\end{split}
\end{equation*}
where we have used \eqref{eqn:pt-pick-2} and $z_i\in B_{g(t_1)}(x_1,\rho_1)$. 

By integrating it from $s=0$ to $s=t_1$, we conclude that 
\begin{equation}\label{eqn:vol-low}
\begin{split}
\left[\mathrm{Vol}_{g(t_1)}\left(B_{g(t_1)}\left(z_i,\frac1{40}\rho_1\right) \right)\right]^{2/n}
&\leq \left[\mathrm{Vol}_{g_0}\left(B_{g(t_1)}\left(z_i,\frac1{40}\rho_1\right) \right)\right]^{2/n} +\frac{2\delta_1}n t_1.
\end{split}
\end{equation}

On the other hand, by assumption (c) and the monotonicity of the $\nu$-entropy (in terms of domains) \cite{Wang2018}*{Propositions 2.1}, $\nu(B_{g(t_1)}(z_i,\frac1{40} \rho_1, g(t_1),1)\geq -A.$ 
Thus, \eqref{eqn:pt-pick-2}, \cite{Martens2025}*{Lemmas 3.1 \& 3.3}, and Remark~\ref{rem:size-delta1} imply 
\begin{equation}\label{eqn:vole-low}
\mathrm{Vol}_{g(t_1)}\left(B_{g(t_1)}\left(z_i,\frac1{40}\rho_1\right) \right)\geq \gamma_0\cdot \rho_1^n
\end{equation}
for some $\gamma_0(n,A)>0$. 

By Claim~\ref{claim:lower-bdd} with $\b=2\gamma_0^{-1/n}$, then we have
$$\gamma_0^{2/n} \rho_1^2\geq  4 t_1> \frac{4}{n} \delta_1 t_1$$
provided that 
$\a_0
\leq C_0^{-2}\delta_1\cdot \min\set{2\gamma_0^{-1/n}, 8^{-2}\tilde T}^{ -2}
=:\gamma_1(n,A) \cdot \delta_1$. 
Using this with \eqref{eqn:vole-low} and \eqref{eqn:vol-low}, we estimate the left hand side in \eqref{eqn:covering} by
\begin{equation*}
\begin{split}
\sum_{i=1}^N \mathrm{Vol}_{g_0}\left(B_{g(t_1)}\left(z_i,\frac1{40}\rho_1\right) \right)
&\geq \sum_{i=1}^N\left(\gamma_0^{2/n} \rho_1^2 -\frac{2\delta_1}{n}t_1\right)^{n/2}\\
&\geq N2^{-n/2} \gamma_0 \rho_1^n
\end{split}
\end{equation*}
so that $N\leq 4^n\gamma_0^{-1} v_0=: N_0(n,A,v_0)$ from \eqref{eqn:covering}.
\end{proof}

We can finish the proof of Proposition~\ref{prop:improve-local-total-pre}.
We now choose $\delta_0(n,v_0,A):= (4N_0)^{-2/n}\delta_1$ so that from Claim~\ref{claim:cover-num}, we can deduce from \eqref{eqn:bound-delta_1-delta_0} that
\begin{equation*}
\rho_1^2<  2C_3N_0  t_1.
\end{equation*}
By Claim~\ref{claim:lower-bdd}, this is impossible if we choose 
$$\a_0:=\delta_1\cdot  
\min\set{(8C_0)^{-2}\tilde T, \gamma_1,(4C_0C_1)^{-2}, {C_0^{-2}} (2C_3N_0)^{-1}}.$$
This proves the claim~\eqref{eqn:pt-picking}. {The proposition follows by the claim using $\rho(x_1,t)=1$.}
\end{proof}

The upshot of Proposition~\ref{prop:improve-local-total-pre} is that it allows us to re-apply Lemma~\ref{lma:local-total-pre} to obtain an improved estimate of curvature.

\begin{prop}\label{prop:improved-pt-esti}
For $n,A>0$, there exists $\Lambda_0(n,A)>0$ such that the following holds. Under the assumptions in Proposition~\ref{prop:improve-local-total-pre}, if, in addition, we have $||\Rm(g_0)||_{L^{n/2}(B_{g_0}(x_0,1))}\leq \e$, then we have 
$$|\Rm(x_0,t)|\leq \Lambda_0 t^{-1}(\e +t^{2/n})$$
for $t\in (0,T\wedge \tilde T]$.
\end{prop}

\begin{proof}
By Proposition~\ref{prop:improve-local-total-pre}, we have $||\Rm(g(t))||_{L^{n/2}(B_{g(t)}(x_0,1))}\leq \delta_1$ for $t\in [0,T\wedge \tilde T]$, where $\delta_1(n,A)$ is the constant from Lemma~\ref{lma:local-total-pre}. 
Applying Lemma~\ref{lma:local-total-pre} to the rescaled Ricci flow $\tilde g(t):= 4g(4^{-1}t),t\in [0,4(T\wedge \tilde T)]$ and then rescaling it back to $g(t)$, we see that
\begin{align*}
\int_{B_{g(t)}(x_0,\frac12 )}|\Rm(g(t))|^{n/2}\, d\mathrm{vol}_{g(t)} 
& \le \int_{B_{g_0}(x_0, 1)}|\Rm(g_0)|^{n/2}\, d\mathrm{vol}_{g(0)} + 4C_1\delta_1^{n/2} t\\
& \leq \e^{n/2}+ 4C_1\delta_1^{n/2} t
\end{align*}
for all $t\in [0,T\wedge\tilde T]$. 
Thanks to the curvature assumption and the volume non-collapsing from \cite{Wang2018}*{Theorem 5.4}, standard Moser iteration (see \cite[Chapter 19]{PeterLi} for example)
yields 
\begin{equation*}
t|\Rm(x_0,t)|\leq \Lambda_0 \pr{\e +t^{2/n}}
\end{equation*}
for some $\Lambda_0(n,A)>0$. 
This completes the proof.
\end{proof}

In particular, Proposition~\ref{prop:improved-pt-esti} asserts that the Ricci flow has an improved estimate from the smallness of the initial local curvature concentration.

\bigskip

The next Lemma states that for a \textit{complete bounded curvature} Ricci flow, the local $\nu$-functional is (almost) preserved. 
This was first considered by Wang \cite{Wang2018}, localizing the work of Perelman \cite{Perelman2002}. 
The following refined version is by Cheng \cite{Cheng2025}.

\begin{lma}[Lemma 2.5 in \cite{Cheng2025}]\label{lma:cheng}
Fix $0\leq s_1\leq s_2\leq 1$. 
Suppose $g(t)$ is a smooth solution to the Ricci flow on $M\times [s_1,s_2]$ with bounded curvature such that for all $t\in [s_1,s_2]$ and $x\in B_{g(t)}(x_0,\sqrt{t}),$
$$\Ric(g(t))\leq (n-1)\hat \a t^{-1}.$$
Then for any $\tau>0$, $\b\in (0,\frac1{10}),$ and $\gamma\in [0,8-20\b)$, we have 
\begin{equation*}
\begin{split}
&\quad \nu(\Omega_2,g(s_2),\tau+1-s_2)-\nu(\Omega_1,g(s_1),\tau+1-s_1)\\
&\geq -\left[ \frac{1+\tau}{10 \hat \a ^2\b^2}+\frac1e\right]\cdot \left[ \exp\left(\frac{s_2-s_1}{10 \hat \a ^2\b^2}\right)-1\right]
\end{split}
\end{equation*}
where 
\begin{equation*}
\left\{
\begin{array}{ll}
\Omega_1:=B_{g(s_1)}\left(x_0,10\hat \a -2\hat \a \sqrt{s_1}-\hat \a \gamma \right);\\[3mm]
\Omega_2:=B_{g(s_2)}\left(x_0,10\hat \a (1-2\b)-2\hat \a \sqrt{s_2}-\hat \a \gamma\right).
\end{array}
\right.
\end{equation*}
\end{lma}

\section{\bf Local smoothing using Ricci flow}\label{sec:RF-exist}

In this section, we will construct local solutions to the Ricci flow which locally regularize metrics with small local curvature concentration, i.e. Theorem~\ref{thm:localRF}.
We adapt the idea of Cheng \cite{Cheng2025} with the estimates in Section~\ref{sec:RF-prel} to construct local Ricci flow, building on the idea of Hochard \cite{Hochard_thesis}. To prove Theorem~\ref{thm:localRF}, we need to specify some constants precisely. 
We need two more ingredients. 
The first is a result of Hochard that allows us to construct a local Ricci flow on regions with bounded curvature.

\begin{prop}[Proposition 4.2 in \cite{LeeTopping2022}, based on \cite{Hochard_thesis}]\label{prop:construction-local-RF-new}
For $n\geq 2$ there exist constants $\gamma(n)\in (0,1]$ and $\Lambda(n)>1$  so that the following is true.
Suppose $(N^n,h_0)$ is a smooth manifold (not necessarily complete)
that satisfies $|\Rm(h_0)|\leq \rho^{-2}$ throughout, for some $\rho>0$.
Then there exists a smooth \textit{complete} Ricci flow $h(t)$ on $N$
for $t\in [0,\gamma \rho^2]$, with the properties that 
\begin{enumerate}
\item[(i)] $h(0)=h_0$ on $N_\rho=\{x\in N: B_{h_0}(x,\rho)\Subset N\}$;
\item[(ii)] $|\Rm(h(t))|\leq  \Lambda \rho^{-2}$ throughout $N\times [0,\gamma \rho^2]$.
\end{enumerate}
\end{prop}

The Ricci flow obtained in \cite[Proposition 4.2]{LeeTopping2022} is constructed by modifying the incomplete metric at its extremities in order to make it complete, and followed by running the classical Shi's solution. Thus, the solution $h(t)$ is by construction a complete solution on the whole $N$.

We also recall the shrinking balls lemma, which is one of the local ball inclusion results based on  the distance distortion estimates of Hamilton and Perelman from  \cite[Lemma 8.3]{Perelman2002}.

\begin{lma}[{\cite[Corollary 3.3]{SimonTopping2022}}]
\label{l-balls}
For $n\geq 2$ there exists a constant $\beta(n)\geq 1$  such that the following is true. 
Suppose $(N^n,g(t))$ is a Ricci flow for $t\in [0,S]$ and $x_0\in N$ with $B_{g(0)}(x_0,r)\Subset N$ for some $r>0$, and $\Ric_{g(t)}\leq a(n-1)/t$ on $B_{g_0}(x_0,r)$ for each $t\in (0,S]$. 
Then for each $t\in (0,S],$ 
$$B_{g(t)}\left(x_0,r-\beta\sqrt{a t}\right)\subseteq B_{g_0}(x_0,r).$$
\end{lma}
\bigskip

We are now ready to prove Theorem~\ref{thm:localRF}.

\begin{proof}[Proof of Theorem~\ref{thm:localRF}]
 We first prove the existence when $||\Rm(g_0)||_{L^{n/2}(B_{g_0}(x_0,1),g_0)}\leq \e_1$ for $\e_1>0$ which is fixed below.
By \cite[Lemma 3.2]{Martens2025}, we might assume $\nu\left(B_{g_0}(x,1),g_0,2 \right)\geq -A_0$ for some $A_0(n,C_s)$, by working on a slightly smaller ball if necessary. 
We need to specify how the positive constants are chosen, since  the precise forms of the constants will play a crucial role in the construction.
\begin{itemize}
\item $\gamma(n)$ and $\Lambda(n)$ from Proposition~\ref{prop:construction-local-RF-new};
\item $\b(n)$ from Lemma~\ref{l-balls};
\item $A:=A_0+1$;
\item $\tilde T(n),\delta_0(n,v_0,A)$ and $\a_0(n,v_0,A)$ from Proposition~\ref{prop:improve-local-total-pre};
\item $\Lambda_0(n,A)$ from Proposition~\ref{prop:improved-pt-esti};
\item $\e_1(n,v_0,A):=\min\set{ \delta_0, \a_0 (2\Lambda_0)^{-1}, \a_0 (4\Lambda\Lambda_0)^{-1}}$;
\item $\mu(n,v_0,A):=\min\set{ 10^{-2},\gamma (2\e_1 \Lambda_0)^{-1}}$;
\item $\sigma(n,v_0,A):=10^{-3} (1-(1+\mu)^{-1/5} )$;
\item $L(n,v_0,A):=10^4\a_0^{-1/2}\cdot \max\set{ \b,\tilde T ^{-1/2},\e_1^{-n}   ,n, \mu\sigma^{-2} , \sigma^{-2}\left[1-(1+\mu)^{-1/5} \right]^{-1}}$;
\item $\hat T(n,v_0,A):=\min\set{ 10^{-1}\a_0^{-1}L^{-2},\left(\frac14 L^2\a_0+\mu+1\right)^{-1}}.$
\end{itemize}

Let $\rho_0\in (0,\frac12)$ be small such that $|\Rm(g_0)|\leq \rho_0^{-2}$ on $N_0:=B_{g_0}(x_0,2+\bar r)$. 
Then from Proposition~\ref{prop:construction-local-RF-new}, there exists a smooth complete Ricci flow $h(t)$ on $N_0\times [0,\gamma \rho_0^2]$ such that 
\begin{itemize}
\item $h(0)=g_0$ on $(N_0)_{\rho_0} := B_{g_0}(x_0,2+\bar r-\rho_0)$;
\item  $|\Rm(h(t))|\leq \Lambda \rho_0^{-2}$ on $N_0\times [0,\gamma \rho_0^2]$.
\end{itemize}
We choose $t_0:= \rho_0^2\cdot \min\{\gamma, \Lambda^{-1}\a_0\}$ so that
\begin{equation*}
|\Rm(h(t))|\leq \Lambda \rho_0^{-2}\leq \a_0 t_0^{-1} \leq \a_0t^{-1}
\end{equation*}
on $N_0\times (0,t_0]$.
We also denote $t_{-1}=0$ for notional consistence. 
We define a sequence of times $t_k$ and radii $r_k$ inductively as follows. 
\begin{enumerate}
\item[(i)] $r_0:=2+\bar r-\rho_0$;
\item[(ii)] $t_{k+1}:=(1+\mu) t_k$ for $k\geq 0$;
\item[(iii)] $r_{k+1}=r_k-8L\sqrt{\a_0 t_k}$  for $k\geq 0$.
\end{enumerate}

We consider the statement $\mathcal{P}(k)$: 
\begin{quote}
There exists a smooth Ricci flow $g(t)$ defined on $B_{g_0}(x_0,r_k)\times [0,t_k]$ such that
\end{quote}
\begin{enumerate}
\item[(I)] $g(0)=g_0$,
\item[(II)] $|\Rm(g(t))|\leq \a_0 t^{-1}$ for $t\in (0,t_k]$, and
\item[(III)] $ \nu\left(B_{ g(t)}\left(x,L\sqrt{\a_0 t_k}\right), g(t),\frac14 L^2 \a_0 t_k\right)\geq -A$ for all $(x,t)\in B_{g_0}(x_0,r_k)\times (0,t_k]$.
\end{enumerate}

From the discussion above, since $h(t)$ is smooth and $A>A_0$, we might assume $\mathcal{P}(0)$ is true by shrinking $\rho_0$ if necessary, for instances using the continuity of entropy \cite[Corollary 2.5]{Wang2018}.  
We want to show that $\mathcal{P}(k)$ is true if $r_k>0$ and $t_k<\hat T$, both of which hold when $k=0$. 

We prove this by induction on $k$. 
$\mathcal{P}(0)$ is true by the construction above. 
{ Suppose that  for some $k\ge 0$ with $r_{k+1}>0$ and $t_{k+1}< \hat T,$ $\mathcal{P}(i)$ is true for $0\le i\le k.$}
We let $g(t)$ be the Ricci flow  on $B_{g_0}(x_0,r_k)\times [0,t_k]$ from $\mathcal{P}(k)$. 
We first improve the curvature estimate on a smaller domain. 

\begin{claim}\label{claim:impro-est}
We have $$|\Rm(g(t))|\leq 2\e_1\Lambda_0t^{-1}$$ on $B_{g_0}(x_0,r_k-2L\sqrt{\a_0 t_k})\times [0,t_k]$.
\end{claim}

\begin{proof}[Proof of Claim~\ref{claim:impro-est}]
Let $z\in B_{g_0}(x_0,r_k-2L\sqrt{\a_0 t_k})$ so that Lemma~\ref{l-balls} implies
\begin{equation*}
\begin{split}
B_{g(t)}(z,L\sqrt{\a_0 t_k})&\subseteq B_{g_0}(z,2L\sqrt{\a_0 t_k})\subseteq  B_{g_0}(x_0,r_k)
\end{split}
\end{equation*}
for all $t\in [0,t_k]$. 
We consider the rescaled Ricci flow $\tilde g(t):=\lambda^{-2}g(\lambda^2 t)$ for $t\in [0,\lambda^{-2}t_k]$ where $\lambda=\frac12 L\sqrt{\a_0 t_k}$ and denote $\tilde g_0:=\tilde g(0)$. 
Then for all $t\in [0,4L^{-2}\a_0^{-1}]$, we have
\begin{enumerate}
\item[(A)] $B_{\tilde g(t)}(z,2)=B_{g(\lambda^2 t)}(z,2\lambda)\Subset M$;
\item[(B)] $|\Rm(\tilde g(t))|\leq \a_0 t^{-1}$ on $B_{\tilde g(t)}(z,2)$;
\item [(C)]
$ \nu\left(B_{\tilde g(t)}(z,2),\tilde g(t),1\right)\geq  -A$;
\item[(D)] $\mathrm{Vol}_{\tilde g_0}\left( B_{\tilde g_0}(x,r)\right)\leq v_0r^n$ for $x\in B_{\tilde g_0}(z,1)$ and $r\in (0,1]$;
\item[(E)] $||\Rm(\tilde g_0)||_{L^{n/2}(B_{\tilde g_0}(z,2) , \tilde g_0)}<\e_1$.
\end{enumerate}
Here we have used $2\lambda=L\sqrt{\a_0 t_k}<1$ so that $B_{\tilde g_0}(x,2)\subseteq B_{g_0}(x,1)$ where the initial estimates hold. 
The entropy lower bound (C) follows from (III):
\begin{equation*}
\begin{split}
\nu\left(B_{\tilde g(t)}(z,2),\tilde g(t),1\right)
&= \nu\left(B_{ g(s)}\left(z,L\sqrt{\a_0 t_k}\right), g(s),\frac14 L^2 \a_0 t_k\right)\geq -A
\end{split}
\end{equation*}
for $s=\lambda^2t\in [0,t_k]$. 
Thus, the rescaled flow $\tilde g(t)$ satisfies all the conditions of Proposition~\ref{prop:improve-local-total-pre}.
It follows from Proposition~\ref{prop:improved-pt-esti} and our choice of $L$ and $\e_1$ that for $t\in  (0,t_k]$,
\begin{equation*}
|\Rm(z,t)|\leq 2\e_1\Lambda_0 t^{-1}.
\end{equation*} 
This proves the Claim.
\end{proof}

We now use the improved estimate from Claim~\ref{claim:impro-est} to extend the solution.

\begin{claim}\label{claim:extension}
The Ricci flow $g(t)$ can be extended to $[0,t_{k+1}]$ on $B_{g_0}(x_0,r_k-4L\sqrt{\a_0 t_k})$ such that (I)-(II) in $\mathcal{P}(k+1)$ hold. 
Moreover, there is a complete bounded curvature Ricci flow $h_k(t)$ on $B_{g_0}(x_0,r_k-2L\sqrt{\a_0t_k})\times [t_k,t_{k+1}]$ such that $|\Rm(h_k(t))|\leq \a_0 t^{-1}$ and  $h_k(t)|_{B_{g_0}(x_0,r_k-4L\sqrt{\a_0 t_k})}=g(t)$.
\end{claim}

\begin{proof}[Proof of Claim~\ref{claim:extension}]
Denote $\a_1:= 2\e_1\Lambda_0\le \alpha_0$ for convenience.  
We apply Proposition~\ref{prop:construction-local-RF-new} with a translation in time, on $N=N_k:=B_{g_0}(x_0,r_k-2L\sqrt{\a_0t_k})$ and $h_{0,k}=g(t_k)$ with $\rho=\rho_k:=\sqrt{\a_1^{-1}t_k}$ to obtain a complete Ricci flow $h_k(t)$ on $N_k\times [t_k,t_k+\gamma \a_1^{-1} t_k]$ such that 
\begin{itemize}
\item $h_k(t_k)=h_{0,k}$ on $(N_k)_{\rho_k}=\{ x \in N_k : B_{g(t_k)}(x,\rho_k)\Subset N_k \}$;
\item $|\Rm(h_k(t))|\leq \Lambda \rho_k^{-2}= \Lambda \a_1 t_k^{-1}$ on $N_k\times [t_k,(1+\gamma \a_1^{-1} )t_k]$.
\end{itemize}

It remains to show that $(N_k)_{\rho_k}\supseteq B_{g_0}(x_0,r_{k}-4L\sqrt{\a_0 t_k})$.  
Let $x\in B_{g_0}(x_0,r_{k}-4L\sqrt{\a_0 t_k})$, and hence we have $B_{g_0}(x,2L\sqrt{\a_1 t_k})\subseteq N_k$. 
Thanks to (II) in the induction hypothesis $\mathcal{P}(k)$, Lemma~\ref{l-balls} implies 
\begin{equation*}
B_{g(t_k)}(x,L\sqrt{\a_0 t_k}) 
\Subset B_{g_0}(x,2L\sqrt{\a_0 t_k})\subseteq N_k
\end{equation*}
so that $x\in (N_k)_{\rho_k}$. 
Thus, $h_k({ t_k})=h_{0,k}=g(t_k)$ on $B_{g_0}(x_0,r_{k}-4L\sqrt{\a_0 t_k})$. 

By our choice of $\mu$, we know $t_{k+1}-t_k=\mu t_k\leq \gamma \a_1^{-1} t_k$. 
Therefore, we  might patch the Ricci flow together on $B_{g_0}(x_0,r_k-4L\sqrt{\a_0t_k})\times [0,t_{k+1}]$ by taking  
\begin{equation*}
g(t):=\begin{cases}
    g(t)&\text{if } t\in[0, t_k]\\
    h_k(t)&\text{if } t\in [t_k,t_{k+1}]
\end{cases}.
\end{equation*}
By our choice of $\e_1$, we see that for $t\in [t_k,t_{k+1}]$,
\begin{equation*}
\begin{split}
|\Rm(g_k(t))|
\leq \Lambda \a_1 t^{-1}_k
& =  2\e_1 \Lambda_0\Lambda(1+\mu)  t^{-1}_{k+1}\\
&\leq  4\e_1 \Lambda_0\Lambda t^{-1}\leq \a_0 t^{-1}.
\end{split}
\end{equation*}
This finishes the proof of the claim.
\end{proof}
{ We will assume $g(t)$ for $t\leq t_k$, is constructed by inductively using Claim~\ref{claim:extension} so that on each sub-interval $[t_i,t_{i+1}]$, the Ricci flow $g(t)$ is restriction of some complete bounded curvature Ricci flow.} 

It remains to show that (III) in $\mathcal{P}(k+1)$ holds. 
For each $i$, from Claim~\ref{claim:extension} 
and the induction hypothesis, there is a complete bounded curvature Ricci flow $h_i(t)$ on $B_{g_0}(x_0,r_i-2L\sqrt{\a_0 t_i})\times [t_i,t_{i+1}]$ such that 
\begin{itemize}
\item $h_i(t)|_{B_{g_0}(x_0,r_i-4L\sqrt{\a_0 t_i})}=g(t)$;
\item $|\Rm(h_i(t))|\leq \a_0 t^{-1}$ where $\a_0< 1<10^3n$.
\end{itemize}

We now consider the rescaled Ricci flow $\tilde g(t):=t_k^{-1}g(t_k t)$ for $t\in [0,1+\mu]$. 
For each $i$, we also define the rescaled Ricci flow $\tilde h_i(t):=t_k^{-1}h_i(t_k t)$ for $t\in [\tilde t_i,\tilde t_{i+1}]$ where {$\tilde t_i:=t_k^{-1}t_i$}. 
For $z\in B_{g_0}(x_0,r_{k+1})$, we apply Lemma~\ref{lma:cheng} to deduce 
\begin{equation}\label{eqn:monot}
\begin{split}
&\quad \nu\left(B_{\tilde h_i(\tilde t_{i+1})}\left(z, \pr{10(1-2\hat\b)- \gamma-2\sqrt{\tilde t_{i+1}}} \hat\a\right),\tilde h_i(\tilde t_{i+1}),\tau+1-\tilde t_{i+1} \right)\\
&\geq \nu\left(B_{\tilde h_i(\tilde t_{i})}\left(z, \pr{10- \gamma-2\sqrt{\tilde t_{i}}} \hat\a\right),\tilde h_i(\tilde t_{i}),\tau+1-\tilde t_{i}\right)\\
&\quad -\left( \frac{1+\tau}{10 \hat \a ^2\hat\b^2}+\frac1e\right)\cdot \left[ \exp\left(\frac{\mu \tilde t_i}{10 \hat \a ^2\hat\b^2}\right)-1\right]
\end{split}
\end{equation}
for any $\hat\a\geq 10^3n$, $\tau>0$, $\hat\b\in (0,\frac1{10})$, and $\gamma\in [0,8-20\hat\b)$. 
The analogous estimate also holds if $\tilde t_{i+1}$ is replaced by any $t\in [\tilde t_i,\tilde t_{i+1}]$. 
\bigskip

We start to prove (III) in $\mathcal{P}(k+1).$ 
Suppose $\tilde t\in [\tilde t_\ell,\tilde t_{\ell+1}]$ for some $-1\le \ell\le k.$
We fix $$\hat\a:=5^{-1}L\sqrt{\a_0}\geq 10^3n\quad\text{and}\quad \tau:=\mu+\frac14 L^2\a_0.$$
For $-1\le i\le \ell,$ we define
\begin{align*}
\hat\b_i&:= \sigma\cdot \tilde t_i^{1/5},\\
\gamma_{i}&:=\gamma_{i-1}+20\hat\b_{i-1},\, \text{and}\\
\nu_i&:=\nu\left(B_{\tilde h_i(\tilde t_{i})}\left(z, \pr{10- \gamma_i-2\sqrt{\tilde t_{i}}}\hat\a\right),\tilde h_i(\tilde t_{i}),\tau+1-\tilde t_{i}\right),
\end{align*}
where $\gamma_i$ is defined inductively starting from the final data below. 
Here $\sigma$ is the constant defined in the beginning of the proof. 
Recall that we also set $t_{-1}:=0.$
We also define the "initial data": 
\begin{align*}
\hat \b_{\tilde t}&:= \sigma\,  \tilde t^{1/5},\\
\gamma_{\tilde t}&:= 10(1-2\hat \b_{\tilde t})-2\sqrt{\tilde t}-L\hat \a^{-1}\sqrt{\a_0},\,\\
 \gamma_{\ell}&:= \gamma_{\tilde t}+20\hat\b_{\tilde t},\, \text{and}\\
\nu_{\tilde t}&:= { \nu\left(B_{\tilde h_{\ell}(\tilde t)}\left(z, \pr{10- \gamma_{\tilde t}-2\sqrt{\tilde t}}\hat\a\right),\tilde h_{\ell}(\tilde t),\tau+1-\tilde t\right).}
\end{align*}
The definition of $\gamma_{\tilde t}$ implies 
\begin{equation}\label{eqn:choice-initial}
10\pr{1-2\hat \b_{\tilde t}}-\gamma_{\tilde t}-2\sqrt{\tilde t}= L\hat \a^{-1}\sqrt{\a_0}=5.
\end{equation}

Since
$B_{g_0}(z,4L\sqrt{\a_0t_k})\Subset B_{g_0}(x_0,r_k-4L\sqrt{\a_0t_k})$
where ${ g(t)= h_i(t)}$ for $t\in [t_i,t_{i+1}]$ { from Claim~\ref{claim:extension} and the remark after it}, Lemma~\ref{l-balls} implies that for $t\in [t_i,t_{i+1}]$,
$$B_{ g(t)}\left( z,3L\sqrt{\a_0t_k} \right)\Subset B_{g_0}(x_0,r_i-4L\sqrt{\a_0t_i})$$
and thus, 
\begin{equation*}
\nu_i=\nu\left(B_{\tilde g(\tilde t_{i})}\left(z, \pr{10- \gamma_i-2\sqrt{\tilde t_{i}}}\hat\a\right),\tilde g(\tilde t_{i}),\tau+1-\tilde t_{i}\right).
\end{equation*}
By the monotonicity of the $\nu$-entropy \cite{Wang2018}*{Propositions 2.1}, thanks to $t_k\leq \hat T$ and our choice of $\hat T$, we have
\begin{equation}\label{eqn:initi-nu}
\begin{split}
\nu_{-1}&\geq \nu\left(B_{\tilde g_0}\left(z, 10\hat\a\right),\tilde g_0,\tau+1\right)\\
&=\nu\left(B_{ g_0}\left(z, 2L\sqrt{\a_0 t_k}\right), g_0,\left(\frac14 L^2\a_0+\mu+1\right)t_k\right)\\
&\geq \nu\left(B_{g_0}(x_0,1),g_0,2) \right)\geq -A_0.
\end{split}
\end{equation}

{ 

We show that $\sigma$ is chosen so that the assumptions on $\hat \b_i$ and $\gamma_i$ in \eqref{eqn:monot} are justified. Clearly, $\hat \b_i\leq  \sigma (1+\mu)^{1/5}\leq 10^{-2}$ if $\sigma\leq 10^{-3}$. We now estimate $\gamma_i$. The upper bound is trivial from decreasing property so that $\gamma_i\leq 5<8-20\hat \b_i$. To estimate the lower bound, we claim that $\gamma_{-1}>0$.
\begin{equation}\label{eqn:est-low-gamma}
    \begin{split}
    \gamma_{\tilde t}-\gamma_{-1}&=\gamma_{\tilde t}-\gamma_{\ell} +\sum^\ell_{i=0} (\gamma_{i}-\gamma_{i-1}) \\
    &\leq 20 \sigma\cdot  \sum_{i=0}^{k} (\tilde t_{k-i})^{1/5}\\
    &\leq 20\sigma \sum_{i=0}^\infty (1+\mu)^{-i/5}= \frac{20\sigma}{1-(1+\mu)^{-1/5}}.
    \end{split}
\end{equation}
From \eqref{eqn:choice-initial}, the definition of $\hat\beta_{\tilde t},$ and $\tilde t\le 1+\mu,$ we have
\begin{equation}
    5\leq 2(1+\mu)^{1/2}+20\sigma (1+\mu)^{1/5}+\gamma_{\tilde t}
\end{equation}
so that \eqref{eqn:est-low-gamma} implies 
\begin{equation}
   \gamma_{i}\geq \gamma_{-1}\geq 5-2(1+\mu)^{1/2}-20\sigma (1+\mu)^{1/5}-\frac{20\sigma}{1-(1+\mu)^{-1/5}}>1
\end{equation}
for all $-1\leq i\leq \ell$ by our choice of $\mu$. 
}

{Without loss of generality, we might assume $\tilde t=\tilde t_{\ell+1}$.} By our choices of $\hat \b_i$ and $\gamma_i$, \eqref{eqn:monot} implies that for all $ -1\leq i\leq \ell$,
\begin{equation*}
\begin{split}
\nu_{i+1}-\nu_i &\geq -\left( \frac{(1+\tau)}{ 10\sigma^{2}\hat \a ^2(\tilde t_i)^{2/5}}+\frac1e\right)\cdot \left[ \exp\left(\frac{\mu (\tilde t_i)^{3/5}}{ 10\sigma^2\hat \a ^2}\right)-1\right].
\end{split}
\end{equation*}
Since $  10^{-1}\mu\hat \a^{-2} 
\sigma^{-2}\tilde t_i^{3/5}  \leq \frac{25}{10}\mu L^{-2}\a_0^{-1}\sigma^{-2} \leq 1$, 
{ using $e^z-1\le 2z$ for $z\in (0,1),$}
we can simplify it as 
\begin{equation*} 
\begin{split}
\nu_{i+1}-\nu_i &\geq -2\left( \frac{10(1+\mu+\frac14 L^2\a_0)}{ \sigma^2  L^2 \a_0(\tilde t_i)^{2/5}}+\frac1e\right)\cdot \left(  \frac{10\mu (\tilde t_i)^{3/5}}{ \sigma^2L^2\a_0}\right) \\
&\geq -10^{4}\sigma^{-2}L^{-2}\a_0^{-1} (1+\mu)^{(i-\ell)/5}.
\end{split}
\end{equation*}
Combining it with \eqref{eqn:initi-nu}, we conclude that 
\begin{equation*}
\begin{split}
\nu_{\tilde t}&\geq \nu_{-1}-10^{4}\sigma^{-2}L^{-2}\a_0^{-1}\sum_{j=0}^{\ell+1} (1+\mu)^{-j/5}\\
&\geq -A_0- \frac{10^{4}\sigma^{-2}}{1-(1+\mu)^{-1/5}}\cdot L^{-2}\a_0^{-1}\\
&\geq -A_0-1
=-A.
\end{split}
\end{equation*}
By \eqref{eqn:choice-initial} and the monotonicity of the $\nu$-entropy \cite{Wang2018}*{Propositions 2.1} again, this proves (III) in $\mathcal{P}(k+1)$ holds. 
By induction, this shows that $\mathcal{P}(k)$ is true whenever $r_k>0$ and $t_k\leq \hat T$.

We let $k_0$ be the maximal $k$ such that $r_k\geq 1+\bar r$, i.e., $r_{k_0+1}<1+\bar r\leq r_{k_0}$. 
We claim that $t_{k_0}$ is uniformly bounded from below. 
To see this, by the construction, we have
\begin{equation*}
    \begin{split}
      1+\bar r&>  r_{k_0+1}=r_0+\sum^{k_0}_{i=0} (r_{i+1}-r_i)\\
      &=2+\bar r -\rho_0-8L\sqrt{\a_0}\cdot \sum^{k_0}_{i=0} \sqrt{t_{k_0-i}}\\
      &\geq \frac32 +\bar r-8L\sqrt{\a_0 t_{k_0}}  \cdot \sum^{\infty}_{i=0} (1+\mu)^{-i/2}.
    \end{split}
\end{equation*}
Thus, $t_{k_0}\geq \tilde T(n,v_0,A)>0$. 
Replacing $\hat T$ with $\hat T\wedge \tilde T$, we prove the existence with (1), (3), and (4) on $B_{g_0}(x_0,1+\bar r)$ using $\mathcal{P}(k_0)$ and the monotonicity of entropy (on scales and domains) after relabelling the constants.  
The injectivity radius lower bound (2) follows from (4) and  the result of Cheeger--Gromov--Taylor  \cite{CheegerGromovTaylor1982}. 
By enlarging $\Lambda_1$ if necessary, the H\"older equivalence of distances follows from the proof of \cite[Claim 3.2]{Martens2025-2}; see also the proof of \cite[Lemma 5.1]{ChanHuangLee2024}.

With the existence of Ricci flow, the general case where $||\Rm(g_0)||_{L^{n/2}(B_{g_0}(x_0,1),g_0)}\leq \e <\e_1$ follows from Lemma~\ref{lma:local-total-pre} and Proposition~\ref{prop:improved-pt-esti} after we shrink $\hat T$ if necessary; see the proof of Proposition~\ref{prop:improved-pt-esti}.
\end{proof}

By letting $\bar r\to+\infty$, we obtain a short-time solution to the Ricci flow on $M\times [0,\hat T]$ with the desired estimates.
\begin{thm}\label{thm:short-time}
For $n\geq 3$ and $v_0,C_s>0$, there exist $\delta_1(n,v_0,C_s),\Lambda_1(n,v_0,C_s),L(n,v_0,C_s),$ and $\hat T(n,v_0,C_s)>0$, such that the following holds. Suppose $(M^n,g_0)$ is a complete non-compact manifold such that for all $x\in M$,
\begin{enumerate}
\item[(a)] $ \mathrm{Vol}_{g_0}\left(B_{g_0}(x,r)\right)\leq v_0 r^n$ for all $r\in (0,1]$;
\item[(b)]  for all $u\in W^{1,2}_0( B_{g_0}(x,1), g_0)$,
$$\left(\int_{ B_{g_0}(x,1)} u^\frac{2n}{n-2}\,d\mathrm{vol}_{g_0}\right)^\frac{n-2}{n}\leq C_s \int_{ B_{g_0}(x,1)} \left(|\nabla u|^2+u^2\right)\, d\mathrm{vol}_{g_0};$$
\item[(c)] $||\Rm(g_0)||_{L^{n/2}(B_{g_0}(x,1),g_0)}\leq \e$ for some $\e\in (0,\delta_1]$. 
\end{enumerate} 
Then there exists a smooth complete Ricci flow solution $g(t)$  on $M\times [0,\hat T]$ such that $g(0)=g_0$ and 
\begin{enumerate}
\item $|\Rm(g(t))|\leq \Lambda_1 \e t^{-1}$;
\item $\mathrm{inj}(g(t))\geq \sqrt{t}$;
\item $||\Rm(g(t))||_{L^{n/2}(B_{g(t)}(x,1/2))}\leq \Lambda_1\e $;
\item $\nu \left(B_{g(t)}(x,2L\sqrt{\hat T},g(t),L^2\hat T \right)\geq -\Lambda_1$
\end{enumerate}
for all $(x,t)\in M\times (0,\hat T]$.
\end{thm}

\begin{proof}
Let $x_0\in M$ and $\bar r_i\to+\infty$. 
By Theorem~\ref{thm:localRF}, there exists a sequence of Ricci flows $g_i(t),t\in [0,\hat T]$ such that (1)-(4) hold for all $x\in B_{g_0}(x_0,1+\bar r_i)\times (0,\hat T]$. 
By \cite[Corollary 3.2]{Chen2009} and \cite[Theorem 14.16]{ChowBook}, $g_i(t)$ sub-converges to a Ricci flow $g(t)$ on $M\times [0,\hat T]$ in the local smooth topology as $i\to+\infty$. 
\end{proof}

\begin{rem}\label{rem:allscale}
In case $\bar r=+\infty$, we only manage to control the small local curvature concentration in the suitable scale; see Proposition~\ref{prop:improve-local-total-pre}. 
In contrast with the bounded curvature case, the small global curvature concentration is preserved under bounded curvature Ricci flow. 
It is unclear to us whether there is any jumping phenomenon of the curvature concentration in the case the curvature is initially unbounded. 
\end{rem}

\section{\bf Compactness of manifolds with bounded curvature concentration}
\label{sec:compactness}

In this section, we use the local smoothing result from Theorem~\ref{thm:localRF} to discuss the compactness of the space of compact manifolds with bounded curvature concentration, i.e., Theorem~\ref{intro:main-2}. 
We start with a simple observation that the collection of manifolds satisfying the assumptions in Theorem~\ref{intro:main-2} is pre-compact with respect to the Gromov--Hausdorff topology.

\begin{lma}\label{lma:doubling}
Suppose $(M,g)$ is a compact manifold such that 
\begin{enumerate}
    \item there exists $C_s>0$ so that for all $u\in W^{1,2}(M)$,
$$\left(\int_M u^\frac{2n}{n-2}\,d\mathrm{vol}_{g} \right)^\frac{n-2}{n}\leq C_s\int_M \left(|\nabla u|^2+ u^2\right) \,d\mathrm{vol}_{g} ;$$
\item $\mathrm{diam}(M,g)\leq D_0$.
\end{enumerate}
Then there exists $v_1(n,C_s,D_0)>0$ such that for all $0<r\leq \mathrm{diam}(M,g)$ and $x\in M$, $$\mathrm{Vol}_g\left(B_g(x,r) \right)\geq v_1 r^n.$$
\end{lma}

\begin{proof}
This follows from \cite[Lemma 6.1]{Ye2015}. 
\end{proof}

\begin{lma}\label{lma:compact}
The collection of compact manifolds satisfying (a), (b), and (d) in Theorem~\ref{intro:main-2} is pre-compact in the Gromov--Hausdorff topology.
\end{lma}
\begin{proof}
By Lemma~\ref{lma:doubling} and (b) in Theorem~\ref{intro:main-2}, a manifold $(M,g)$ in the collection is uniformly volume doubling. 
The conclusion then follows from Gromov's compactness Theorem \cite{Gromov1999}.
\end{proof}

The basic idea in proving the compactness without Ricci geometry is based on finding a local homeomorphism from an open ball around a regular point to an open ball in a smooth manifold. 
This is based on mollifying the initial metrics locally using Theorem~\ref{thm:localRF} to a metric with regularity.

\begin{proof}[Proof of Theorem~\ref{intro:main-2}]
We will use $g_{i,0}$ to denote $g_i$ to emphasize that it is the initial metric whenever smoothing is used. 
We closely follow the argument in \cite{Nakajima1988} with modifications using ideas from \cite{SimonTopping2021} and the new smoothing result, i.e., Theorem~\ref{thm:localRF}. 
It is more convenient to first obtain a convergent sequence in metric geometry. 
By Lemma~\ref{lma:compact}, we may assume $(M_i,g_{i,0})$ converges to a metric space $(X,d_\infty)$ in the Gromov--Hausdorff topology after passing to subsequence. 
We also let $\varphi_i: (M_i,d_{g_{i,0}})\to (X,d_\infty)$ be a sequence of $\e_i$-Gromov--Hausdorff approximations for some $\e_i\to 0$ as $i\to+\infty$.

We let $\delta_1(n,v_0,C_s)$ be the constant from Theorem~\ref{thm:localRF}. 
We define the singular set $\mathcal{S}_r$ of scale $r>0$ by 
\begin{align*}
	\cS_r
	:= \set{
		x\in X:
		\liminf_{i\to\infty} ||\Rm(g_{i,0})||_{L^{n/2}(B_{g_{i,0}}(x_i,r))}\ge \delta_1 
		 \text{ for any }x_i\in M_i \text{ such that } d_\infty(\varphi_i(x_i), x)\to 0
	}
\end{align*}
and we consider
\begin{align*}
	\cS:=\bigcap_{r>0}\cS_r.
\end{align*}

\begin{claim}\label{claim:S-finite}
The set $\cS$ is a finite set. 
That is, $\mathcal{S}=\{p_i\}_{i=1}^N$ for some $N\leq (2\Lambda_0\delta_1^{-1})^{n/2}$. In particular, $\mathcal{S}$ is empty if $\Lambda_0<\delta_1/2$.
\end{claim}

\begin{proof}[Proof of Claim~\ref{claim:S-finite}]
For any fixed $r>0,$ we take a collection $\{x^j\}_{j\in J}\subseteq\cS$  such that $\set{B_{d_\infty}(x^j, 2r):j\in J}$ is a disjoint collection of balls satisfying
\begin{align*}
	\cS\sbst \bigcup_{j\in J} B_{d_\infty}\pr{x^j, 3r}.
\end{align*}
For each $j\in J$, $x^j\in \mathcal{S}_r$ for $r>0$. 
We can find an approximating sequence $x^j_i\in M_i$ of $x^j$ such that $d_\infty\pr{\varphi_i(x_i),x}\to 0$ and for large enough $i,$ we have
\begin{align*}
	\left(\int_{B_{g_i}\pr{x^j_i,r}} |\Rm(g_{i,0})|^{n/2} d\mathrm{vol}_{g_{i,0}} \right)^{2/n} > \frac12\delta_1 
\end{align*}
By the disjoint property in $X$ and the approximating property of $x^j_i$'s, we know that for large $i,$
$\set{B_{g_i}(x_i^j, r):j\in J}$ is a disjoint collection in $M_i.$
Hence, in particular, we derive
\begin{align*}
	|J|
	\le (2\Lambda_0 \delta_1^{-1})^{n/2}
\end{align*}
where $\Lambda_0$ is from the assumption of Theorem~\ref{intro:main-2}.
This bound is independent of $r$ and it particularly implies
\begin{align*}
	|\cS|\le   (2\Lambda_0 \delta_1^{-1})^{n/2}.
\end{align*}
This proves the claim.
\end{proof}

To show that the topological space $M$ induced by $(X\setminus \mathcal{S},d_\infty)$ is in fact a manifold with bi-H\"older charts, we must show that given an arbitrary $x\in X\setminus\mathcal{S}$, there is a neighborhood of $x$ that is bi-H\"older homeomorphic to  an open subset of a complete Riemannian manifold. 
We follow the treatment in \cite{SimonTopping2021}. 

Let $x\in X\setminus\mathcal{S}$. 
Then there exist $r>0$ and $x_i\in M_i$ such that $d_\infty(\varphi_i(x_i),x)\to 0$ and 
\begin{equation*}
||\Rm(g_{i,0})||_{L^{n/2}(B_{g_{i,0}}(x_i,r))}< \delta_1
\end{equation*}
for large enough $i.$ 
By considering $\tilde g_{i,0}:=(3r)^{-2}g_{i,0}$, we might assume $r=3$. 
By Theorem~\ref{thm:localRF} with $\bar r=0$, there exists a Ricci flow $g_i(t)$ on $B_{g_{i,0}}(x_i,1)\times [0,\hat T]$ such that $g_i(0)=g_{i,0}$ and 
 \begin{enumerate}
\item $|\Rm(g_i(t))|\leq \Lambda_1 \delta_1 t^{-1}$;
\item $\mathrm{inj}(g_i(t))\geq \sqrt{t}$;
\item $||\Rm(g_i(t))||_{L^{n/2}(B_{g_i(t)}(x_i,1/2))}\leq \Lambda_1 \e $;
\item $\nu \left(B_{g_i(t)}(x_i,2L\sqrt{\hat T},g_i(t),L^2\hat T \right)\geq -\Lambda_1$
\end{enumerate}
for all $(x,t)\in B_{g_{i,0}}(x_i,1)\times (0,\hat T]$. Moreover, we might assume that for some $C_1(n,v_0,C_s)>0,\gamma(n,v_0,C_s)\in (0,1),$ we have 
\begin{equation}\label{eqn:dist-comp}
    C_1^{-1} \left(d_{g_i(t)}(x,y)\right)^{1/\gamma}\leq d_{g_{i,0}}(x,y)\leq C_1  \left(d_{g_i(t)}(x,y)\right)^{\gamma}
\end{equation}
for all $x,y\in B_{g_{i,0}}(x_i,1)$ and $t\in [0,\hat T]$.

We now have all necessary ingredients to carry out the proof in \cite[Theorem 1.4 \& Corollary 1.5]{SimonTopping2021}. 
We only give a sketch. 
We use Shi's estimates and apply smooth compactness on $g_i(\hat T)$ to obtain a family of local diffeomorphisms $F_i$'s such that $F_i^*g_i(t)$ sub-converges to $g_\infty(t)$ on $B_{g_\infty(\hat T)}(x_\infty,\sigma)$ for $\sigma\leq C_1^{-1}\sigma^{-\gamma}$. 
Using $F_i$, we might use the argument in the proof of \cite[Theorem 1.4]{ChanChenLee2022} to show that $F_i^*d_{g_{i,0}}$ sub-converges to a distance function $\tilde d_\infty$ on $B_{g_\infty(\hat T)}(x_\infty,\sigma)$ which is locally bi-H\"older to the distance $d_{g_\infty(\hat T)}$ thanks to \eqref{eqn:dist-comp}. 
The topological manifold structure nearby $x$ follows from the uniqueness of the pointed Gromov--Hausdorff limits, which identifies $\tilde d_\infty$ nearby $x_\infty$ with $d_\infty$ nearby $x$.
\end{proof}

\begin{rem}
From Claim~\ref{claim:S-finite},  when $\Lambda_0$ is sufficiently small, then $\mathcal{S}$ is empty. In this case from the proof of Theorem~\ref{intro:main-2}, $g_i(t)$ exists globally on $M_i\times [0,\hat T]$. In particular, we can construct $X$ by applying smooth compactness on $(M_i,g_i(\hat T))$ so that $X$ admits a smooth structure. 
\end{rem}

\begin{rem}\label{rem:X-smooth-mfd}
Using the H\"older equivalence \eqref{eqn:dist-comp}, it should not be difficult to show that $X\setminus \mathcal{S}$ in Theorem~\ref{intro:main-2} is a $C^\a$ manifold for some $\a\in (0,1)$. 
In general, we conjecture that $X\setminus \mathcal{S}$ is always a smooth manifold. 
\end{rem}

\section{\bf Application to topological gap theorems}
\label{sec:gap}

In this section, we will use  Theorem~\ref{thm:localRF} to prove the gap theorem in differentiable structure of manifolds with small curvature concentration, i.e., Theorem~\ref{intro:main}.  
We start with the long-time existence. 
If the initial metric is assumed to have bounded curvature, this was obtained earlier by \cite{ChauMarten2024}; see also \cite{Chen2022}.

\begin{thm}\label{thm:immortal-RF}
For $n\geq 3$ and $C_s>0$, there exist $\delta_0(n,C_s),\Lambda(n,C_s)>0$ such that  if  $(M^n,g_0)$ is a complete non-compact manifold  satisfying (a)-(c) in Theorem~\ref{intro:main}, then there exists a unique complete Ricci flow $g(t)$ on $M\times [0,+\infty)$ such that $g(0)=g_0$ and 
\begin{enumerate}
\item[(i)] $|\Rm(g(t))|\leq  \Lambda\e t^{-1}$;
\item[(ii)]  $\mathrm{inj}(g(t))\geq \sqrt{ t}$;
\item[(iii)] $||\Rm(g(t))||_{L^{n/2}(M,g(t))}\leq \e$;
\item[(iv)] $\nu(M,g(t))\geq -\Lambda$.
\end{enumerate}
Furthermore, we have $\sup_M|\Rm(g(t))|=o(1)/t$ as $t\to +\infty$. 
\end{thm}
\begin{proof}
By \cite[Lemma 3.3]{Martens2025}, the global entropy is bounded from below. 
By applying Theorem~\ref{thm:short-time} on $g_{i,0}:=R_i^{-2}g_0$ for some $R_i\to+\infty$ and rescaling it back, we obtain a sequence of Ricci flows $g_i(t),t\in [0,R_i^2\hat T],$ satisfying the desired properties. 
By \cite{Lee2025}, $g_i(t)=g_{i+k}(t)$ for all $t\in [0, \hat T R_i] $ and $k\in \mathbb{N}$. 
Thus, we obtain a long-time solution $g(t)$ with (i), (ii), and (iv). 
It also satisfies a rough form of (iii) that the upper bound is $\Lambda \e$. 
By Lemma~\ref{lma:local-total-pre}, we recover the sharp upper bound $\e$.  
The curvature estimate as $t\to +\infty$ follows from \cite[Theorem 1.2]{ChauMarten2024}.
\end{proof}

\begin{rem}
In contrast to the bounded curvature case studied in \cite{ChauMarten2024}, we require volume growth to be at most Euclidean. 
It is unclear whether the volume growth assumption is necessary.
\end{rem}

Using Theorem~\ref{thm:immortal-RF}, the diffeomorphic gap theorem is immediate from the Ricci flow existence and its estimates.

\begin{proof}[Proof of Theorem~\ref{intro:main}]
This follows from Theorem~\ref{thm:immortal-RF} and \cite[Theorem 1.1]{HuangPeng2025}. 
See also \cite[Theorem 1.1]{HeLee2021} and  \cite[Theorem 5.7]{Wang2020}.
\end{proof}

Basic examples of manifolds satisfying the assumptions in Theorem~\ref{intro:main} are compact perturbations of Euclidean metrics. 
By the rigidity case in positive mass Theorems, we observe that the additional condition on scalar curvature $\mathcal{R}\geq 0$ is strong enough to rule them out. 
Motivated by this,  we conjecture: 
\begin{conj}
For any $C_s,v_0>0$ and $n\geq 3$, there exists $\delta_0(n,C_s,v_0)>0$ such that the following holds. Suppose $(M^n,g_0)$ is a complete non-compact manifold such that  
\begin{enumerate}
\item[(a)] $C_{sob}(M)\geq C_s^{-1}$;
\item[(b)] $\mathrm{Vol}_{g_0}\left(B_{g_0}(x,r) \right)\leq v_0r^n$ for all $(x,r)\in M\times \mathbb{R}_+$;
\item[(c)] $||\Rm(g_0)||_{L^{n/2}(M)}\leq \delta_0(n,C_s,v_0)$;
\item[(d)]$\mathcal{R}(g_0)\geq 0$.
\end{enumerate}
Then $(M,g_0)$ is isometric to $\mathbb{R}^n$.
\end{conj}

\begin{rem}\label{rem:positive-R}
When (d) is strengthened to $\Ric(g_0)\geq 0$, the conjecture was proved by Chan and the first named author \cite{ChanLee2024}; see also \cite{Martens2025}. 
This is based on using the long-time behavior of the Ricci flow to show that the tangent cone of $M$ is isometric to Euclidean space and thus $M$ must be  Euclidean space by Colding's volume continuity \cite{Colding1997}. 
Using the same circle of ideas, one should still be able to show that a tangent cone of $M$ exists by \cite[Theorem 1.4]{ChanChenLee2022} and the blow-down of $(M,g_0)$ sub-converges to $\mathbb{R}^n$ in the $d_p$-sense, introduced in \cite{LeeNaberNeumayer2023}, using the method in \cite[Theorem 1.7]{LeeNaberNeumayer2023}. 
It is however unclear how this asymptotic information constrains the original manifold $M$.
\end{rem}

\end{document}